\documentclass[10pt]{article}
\usepackage[a4paper,margin=.8in]{geometry}
\usepackage[T1]{fontenc}
\usepackage[utf8]{inputenc}
\usepackage{lmodern}
\usepackage{microtype}
\usepackage{amsmath,amssymb,amsthm,mathtools}
\usepackage[mathlines,pagewise]{lineno}
\usepackage{aliascnt}
\usepackage{enumitem}
\usepackage{titling}
\usepackage[margin=1cm]{caption}
\usepackage[hang,flushmargin]{footmisc}
\usepackage[hidelinks]{hyperref}
\hypersetup{
  pdftitle={Proper \{a,b\}-edge-weightings of trees},
  pdfauthor={Péter Madarasi},
  pdfsubject={Neighbor-sum-distinguishing edge-weightings of trees},
  pdfkeywords={edge-weighting, tree, perfect matching, linear-time algorithm}
}
\usepackage[nameinlink,noabbrev,capitalise,sort&compress]{cleveref}
\usepackage[yyyymmdd]{datetime}
\usepackage{tikz}
\usetikzlibrary{arrows.meta}
\makeatletter
\edef\ReferenceTool@CleverefVersion{\csname ver@cleveref.sty\endcsname}%
\def\ReferenceTool@BrokenCleverefVersion{2018/03/27 v0.21.4 Intelligent cross-referencing}%
\ifx\ReferenceTool@CleverefVersion\ReferenceTool@BrokenCleverefVersion
  \let\cpageref\relax
  \DeclareRobustCommand{\cpageref}{%
    \@ifstar{\@crefstar{cpageref}}{\@cref{cpageref}}}%
  \let\Cpageref\relax
  \DeclareRobustCommand{\Cpageref}{%
    \@ifstar{\@crefstar{Cpageref}}{\@cref{Cpageref}}}%
  \providecommand*{\@setcpagerefrange}[3]{%
    \@@setcpagerefrange{#1}{#2}{cref}{#3}}%
  \providecommand*{\@setCpagerefrange}[3]{%
    \@@setcpagerefrange{#1}{#2}{Cref}{#3}}%
  \providecommand*{\@setlabelcpagerefrange}[3]{%
    \@@setcpagerefrange{#1}{#2}{labelcref}{#3}}%
\fi
\makeatother

\AtBeginDocument{%
}

\newcommand{\DeclareEquationCrefFormat}[1]{%
  \crefformat{#1}{##2\textup{(##1)}##3}%
  \Crefformat{#1}{##2\textup{(##1)}##3}%
  \crefrangeformat{#1}{##3\textup{(##1)}##4\nobreakdash--##5\textup{(##2)}##6}%
  \Crefrangeformat{#1}{##3\textup{(##1)}##4\nobreakdash--##5\textup{(##2)}##6}%
  \crefmultiformat{#1}%
    {##2\textup{(##1)}##3}%
    { and~##2\textup{(##1)}##3}%
    {, ##2\textup{(##1)}##3}%
    {, and~##2\textup{(##1)}##3}%
  \Crefmultiformat{#1}%
    {##2\textup{(##1)}##3}%
    { and~##2\textup{(##1)}##3}%
    {, ##2\textup{(##1)}##3}%
    {, and~##2\textup{(##1)}##3}%
  \crefrangemultiformat{#1}%
    {##3\textup{(##1)}##4\nobreakdash--##5\textup{(##2)}##6}%
    { and~##3\textup{(##1)}##4\nobreakdash--##5\textup{(##2)}##6}%
    {, ##3\textup{(##1)}##4\nobreakdash--##5\textup{(##2)}##6}%
    {, and~##3\textup{(##1)}##4\nobreakdash--##5\textup{(##2)}##6}%
  \Crefrangemultiformat{#1}%
    {##3\textup{(##1)}##4\nobreakdash--##5\textup{(##2)}##6}%
    { and~##3\textup{(##1)}##4\nobreakdash--##5\textup{(##2)}##6}%
    {, ##3\textup{(##1)}##4\nobreakdash--##5\textup{(##2)}##6}%
    {, and~##3\textup{(##1)}##4\nobreakdash--##5\textup{(##2)}##6}%
  \labelcrefformat{#1}{##2\textup{(##1)}##3}%
  \labelcrefrangeformat{#1}{##3\textup{(##1)}##4\nobreakdash--##5\textup{(##2)}##6}%
  \labelcrefmultiformat{#1}%
    {##2\textup{(##1)}##3}%
    { and~##2\textup{(##1)}##3}%
    {, ##2\textup{(##1)}##3}%
    {, and~##2\textup{(##1)}##3}%
  \labelcrefrangemultiformat{#1}%
    {##3\textup{(##1)}##4\nobreakdash--##5\textup{(##2)}##6}%
    { and~##3\textup{(##1)}##4\nobreakdash--##5\textup{(##2)}##6}%
    {, ##3\textup{(##1)}##4\nobreakdash--##5\textup{(##2)}##6}%
    {, and~##3\textup{(##1)}##4\nobreakdash--##5\textup{(##2)}##6}%
}
\DeclareEquationCrefFormat{equation}
\DeclareEquationCrefFormat{subequation}

\newcommand{\DeclareItemCrefFormat}[1]{%
  \crefformat{#1}{##2\textup{##1}##3}%
  \Crefformat{#1}{##2\textup{##1}##3}%
  \crefrangeformat{#1}{##3\textup{##1}##4\nobreakdash--##5\textup{##2}##6}%
  \Crefrangeformat{#1}{##3\textup{##1}##4\nobreakdash--##5\textup{##2}##6}%
  \crefmultiformat{#1}%
    {##2\textup{##1}##3}{ and~##2\textup{##1}##3}%
    {, ##2\textup{##1}##3}{, and~##2\textup{##1}##3}%
  \Crefmultiformat{#1}%
    {##2\textup{##1}##3}{ and~##2\textup{##1}##3}%
    {, ##2\textup{##1}##3}{, and~##2\textup{##1}##3}%
  \crefrangemultiformat{#1}%
    {##3\textup{##1}##4\nobreakdash--##5\textup{##2}##6}%
    { and~##3\textup{##1}##4\nobreakdash--##5\textup{##2}##6}%
    {, ##3\textup{##1}##4\nobreakdash--##5\textup{##2}##6}%
    {, and~##3\textup{##1}##4\nobreakdash--##5\textup{##2}##6}%
  \Crefrangemultiformat{#1}%
    {##3\textup{##1}##4\nobreakdash--##5\textup{##2}##6}%
    { and~##3\textup{##1}##4\nobreakdash--##5\textup{##2}##6}%
    {, ##3\textup{##1}##4\nobreakdash--##5\textup{##2}##6}%
    {, and~##3\textup{##1}##4\nobreakdash--##5\textup{##2}##6}%
  \labelcrefformat{#1}{##2\textup{##1}##3}%
  \labelcrefrangeformat{#1}{##3\textup{##1}##4\nobreakdash--##5\textup{##2}##6}%
  \labelcrefmultiformat{#1}%
    {##2\textup{##1}##3}{ and~##2\textup{##1}##3}%
    {, ##2\textup{##1}##3}{, and~##2\textup{##1}##3}%
  \labelcrefrangemultiformat{#1}%
    {##3\textup{##1}##4\nobreakdash--##5\textup{##2}##6}%
    { and~##3\textup{##1}##4\nobreakdash--##5\textup{##2}##6}%
    {, ##3\textup{##1}##4\nobreakdash--##5\textup{##2}##6}%
    {, and~##3\textup{##1}##4\nobreakdash--##5\textup{##2}##6}%
}
\DeclareItemCrefFormat{enumi}
\DeclareItemCrefFormat{enumii}
\DeclareItemCrefFormat{enumiii}
\DeclareItemCrefFormat{enumiv}

\theoremstyle{plain}
\newtheorem{theorem}{Theorem}

\newaliascnt{lemma}{theorem}
\newtheorem{lemma}[lemma]{Lemma}
\aliascntresetthe{lemma}

\newaliascnt{proposition}{theorem}

\aliascntresetthe{proposition}

\newaliascnt{corollary}{theorem}
\newtheorem{corollary}[corollary]{Corollary}
\aliascntresetthe{corollary}

\newaliascnt{claim}{theorem}

\aliascntresetthe{claim}

\theoremstyle{definition}
\newaliascnt{definition}{theorem}

\aliascntresetthe{definition}

\theoremstyle{remark}
\newaliascnt{remark}{theorem}

\aliascntresetthe{remark}

\newcommand{\DeclareNamedCrefType}[3]{%
  \crefname{#1}{#2}{#3}%
  \Crefname{#1}{#2}{#3}%
  \crefrangelabelformat{#1}{##3##1##4\nobreakdash--##5##2##6}%
}

\DeclareNamedCrefType{section}{Section}{Sections}
\DeclareNamedCrefType{subsection}{Section}{Sections}
\DeclareNamedCrefType{subsubsection}{Section}{Sections}
\DeclareNamedCrefType{appendix}{Appendix}{Appendices}
\DeclareNamedCrefType{figure}{Figure}{Figures}
\DeclareNamedCrefType{table}{Table}{Tables}
\DeclareNamedCrefType{footnote}{Footnote}{Footnotes}
\DeclareNamedCrefType{page}{Page}{Pages}

\DeclareNamedCrefType{theorem}{Theorem}{Theorems}
\DeclareNamedCrefType{lemma}{Lemma}{Lemmas}
\DeclareNamedCrefType{proposition}{Proposition}{Propositions}
\DeclareNamedCrefType{corollary}{Corollary}{Corollaries}
\DeclareNamedCrefType{claim}{Claim}{Claims}
\DeclareNamedCrefType{definition}{Definition}{Definitions}
\DeclareNamedCrefType{remark}{Remark}{Remarks}

\makeatletter
\@onlypreamble\DeclareEquationCrefFormat
\@onlypreamble\DeclareItemCrefFormat
\@onlypreamble\DeclareNamedCrefType
\makeatother

\newcommand{\R}{\mathbb{R}}
\newcommand{\Q}{\mathbb{Q}}

\usepackage[textsize=tiny]{todonotes}

\newcommand{\Cfam}{\mathcal{C}}
\newcommand{\degthree}{D_3}
\newcommand{\side}{\operatorname{side}}

\title{
  \vspace*{-25pt}
  Proper $\{a,b\}$-edge-weightings of trees
}
\author{
  P\'eter Madarasi\thanks{HUN-REN Alfr\'ed R\'enyi Institute of Mathematics, Re\'altanoda u.\ 13--15, Budapest H-1053, Hungary; and Department of Operations Research, ELTE E\"otv\"os Lor\'and University, P\'azm\'any P.\ s.\ 1/c, Budapest H-1117, Hungary. E-mail: \texttt{madarasi@renyi.hu}}
}
\date{\vspace{-.5cm}}

\begin{document}
\maketitle

\begin{abstract}
  Let $a$ and $b$ be distinct real weights.
  An $\{a,b\}$-edge-weighting of a tree assigns one of these weights to each edge and is proper if adjacent vertices have different sums of incident edge weights.
  For every such pair, we give an explicit structural characterization of the trees that do not admit a proper $\{a,b\}$-edge-weighting.
  If $ab(a+b)\neq0$, then $K_2$ is the only tree without such a weighting.
  If $a+b=0$, then a tree has no proper $\{a,b\}$-edge-weighting exactly when every vertex has degree $1$ or $3$ and the subgraph induced by the degree-$3$ vertices has a perfect matching.
  For the remaining case $ab=0$, form the spanning forest consisting of the edges whose deletion leaves two odd-order components.
  A tree $T$ has no proper $\{a,b\}$-edge-weighting exactly when both bipartition classes have odd order and every component of this forest satisfies two conditions.
  First, every component satisfies the preceding degree-and-matching condition.
  Second, within each component, the degree of a vertex $v$ in the forest plus twice the number of incident edges $e$ outside the forest for which the component of $T-e$ not containing $v$ has an odd number of vertices from each bipartition class is independent of $v$.
  For every fixed pair of distinct real weights, the proofs yield a linear-time algorithm that decides whether a proper $\{a,b\}$-edge-weighting exists and constructs one when it does.

\end{abstract}

\section{Introduction}\label{sec:introduction}

For a set $W$ of edge weights, a \emph{$W$-edge-weighting} of a graph $G$ is a map $\omega:E(G)\to W$.
For $v\in V(G)$, write $N_G(v)$ and $E_G(v)$ for the sets of neighbors and incident edges of $v$, respectively, and write $d_G(v)=|N_G(v)|$.
The weighting $\omega$ induces the \emph{weighted degree} $s_\omega(v)=\sum_{e\in E_G(v)}\omega(e)$ at each vertex $v$.
The weighting is \emph{proper} if $s_\omega(u)\neq s_\omega(v)$ for every edge $uv$; equivalently, the weighted degrees form a proper vertex coloring.
Karo\'nski, \L{}uczak, and Thomason initiated this line of research with the $1$-$2$-$3$ conjecture, which asserts that every graph without an isolated edge admits a proper $\{1,2,3\}$-edge-weighting~\cite{karonski2004edge}.
Keusch proved the conjecture in 2024~\cite{keusch2024solution}.
We study the corresponding existence problem for a fixed pair of distinct real weights on trees.
We complete the classification for all such pairs and give direct structural descriptions of the trees that admit no proper weighting.

To state the two exceptional cases, for a graph $G$ let $\degthree(G)$ denote the subgraph induced by the degree-$3$ vertices in $G$.
We use the convention that the empty graph has a perfect matching.
Let $\Cfam$ be the family of trees $H$ such that every vertex of $H$ has degree $1$ or $3$ and $\degthree(H)$ has a perfect matching.
For a tree $T$, the \emph{odd-cut skeleton} $J(T)$, or \emph{skeleton} for short, is the spanning forest consisting of the edges whose deletion leaves two odd-order components.
Because $J(T)$ is spanning, a vertex incident with no skeleton edge is an isolated vertex of $J(T)$ and has degree $0$.
Suppose that $T$ has bipartition $X\cup Y$ with $|X|$ and $|Y|$ odd.
For every nonskeleton edge, the two components obtained by deleting the edge have even order, and exactly one contains an odd number of vertices from each of $X$ and $Y$.
Orient the edge toward this \emph{odd--odd component}.
Let $\ell_T(v)$ be the number of nonskeleton edges directed away from $v$, and define the \emph{adjusted load} $\lambda_T(v)=d_{J(T)}(v)+2\ell_T(v)$.
The adjusted load assigns one unit to each incident skeleton edge and two units to each nonskeleton edge directed away from $v$.
Our main result is the following.
\begin{theorem}\label{thm:main}
  Let $a,b\in\R$ be distinct, and let $T$ be a tree.
  The following statements characterize when $T$ admits no proper $\{a,b\}$-edge-weighting.
  \begin{enumerate}[label=(\roman*),ref=(\roman*),leftmargin=*]
    \item\label{case:generic} If $ab(a+b)\neq0$, then $T$ admits no proper $\{a,b\}$-edge-weighting if and only if $T\cong K_2$.
    \item\label{case:opposite} If $a+b=0$, then $T$ admits no proper $\{a,b\}$-edge-weighting if and only if $T\in\Cfam$.
    \item\label{case:zero} If $ab=0$, then $T$ admits no proper $\{a,b\}$-edge-weighting if and only if both bipartition classes have odd order and, for every component $H$ of $J(T)$, the tree $H$ belongs to $\Cfam$ and $\lambda_T$ is constant on $V(H)$.
  \end{enumerate}
\end{theorem}

The theorem reveals a sharp distinction between the generic regime and the two exceptional regimes.
It also exposes a structural link between the exceptional regimes, since every component of the odd-cut skeleton of a zero-weight obstruction is itself an opposite-weight obstruction.
The zero-weight case additionally requires both bipartition classes to have odd order and the adjusted load to be constant on each skeleton component.
Because $a\neq b$, exactly one of the three cases in \cref{thm:main} applies.
We call them the \emph{generic regime}, \emph{opposite-weight regime}, and \emph{zero-weight regime}, respectively.
For a fixed tree, exchanging the two weights or multiplying both by the same nonzero scalar does not affect the existence of a proper weighting.
Thus the opposite-weight and zero-weight regimes may be normalized to $\{-1,1\}$ and $\{0,1\}$, respectively.
The three regimes require different arguments.
In the generic regime, the proof uses a leaf-extension argument.
In the opposite-weight regime, a recurrence for the possible weighted degrees at the neighbor of a fixed leaf yields the perfect-matching condition.
In the zero-weight regime, prescribing the parity of every weighted degree isolates the odd-cut skeleton, while a charging argument on its components reduces the remaining constraints to the opposite-weight classification.
Across the three regimes, the proofs yield algorithms that, for every fixed pair of distinct weights, decide whether a proper weighting exists and construct one in $O(n)$ time using $O(n)$ space.
An implementation of these algorithms for rational weights is available at \url{https://github.com/madarasip/ab-edge-weighting-tree}.
We next place these results in the context of earlier work.

\medskip
Khatirinejad, Naserasr, Newman, Seamone, and Stevens proved that every tree with at least three vertices has a proper $\{a,b\}$-edge-weighting when $a$ and $b$ have the same sign~\cite[Corollary~2.7]{khatirinejad2012twoedgeweights}.
Their results also imply the same conclusion for nonzero pairs with irrational ratio by transferring a weighting for a same-sign pair to the target pair~\cite[Proposition~2.1(ii) and Corollary~2.7]{khatirinejad2012twoedgeweights}.
The only generic pairs not covered are therefore the mixed-sign pairs with rational ratio and nonzero sum.
\Cref{thm:generic} settles this remaining range.

For the opposite pair $\{-1,1\}$, Bensmail, Mc Inerney, and Lyngsie gave a recursive description of the trees without a proper $\{-1,1\}$-edge-weighting~\cite[Theorem~1.2]{bensmail2022oddab}.
\Cref{thm:opposite} gives an independent direct characterization in terms of the perfect-matching condition on the subgraph induced by the degree-$3$ vertices.
Together, \cref{thm:generic,thm:opposite} recover the complete tree classification in~\cite{bensmail2022oddab} for every pair $\{a,a+2\}$ where $a$ is an odd integer.

Lyngsie gave a recursive description of all trees without a proper $\{0,1\}$-edge-weighting and a polynomial-time recognition algorithm~\cite[Theorem~4 and Section~3]{lyngsie2018nsd01}.
\Cref{thm:odd-cut-skeleton} gives a direct characterization in terms of the odd-cut skeleton, the family $\Cfam$, and the adjusted load.
\Cref{cor:odd-cut-matching-characterization} records the equivalent degree, matching, and load conditions.
Thus, in both exceptional regimes, our results replace previously known recursive descriptions with directly checkable structural conditions.

A related algorithmic result concerns prescribed partial weightings.
For each fixed pair of distinct integer weights, Madarasi and Simon proved that the problem of deciding whether such a partial weighting extends to a proper weighting is polynomial-time solvable on trees~\cite[Theorem~3.3]{madarasisimon2022ab}.
This extension problem is distinct from the existence problem considered here, in which no edge weights are prescribed.

\medskip
We use the following conventions throughout.
All graphs are finite and simple, and the one-vertex tree is allowed.
The graph $K_2$ admits no proper weighting for any pair of distinct weights, whereas the one-vertex tree vacuously admits a proper weighting.
In all algorithmic statements, the weights $a$ and $b$ are fixed constants.
Fix a representation of weighted degrees as follows.
If $a$ and $b$ are linearly independent over $\Q$, represent the weighted degree of a vertex by the ordered pair consisting of the numbers of its incident edges of weights $a$ and $b$.
Otherwise, write $(a,b)=c(p,q)$ for a nonzero real number $c$ and fixed coprime integers $p,q$, and represent the weighted degree $xa+yb$ by the integer $px+qy$.
Equality of represented values is equivalent to equality of the corresponding weighted degrees, and every represented value uses $O(\log n)$ bits.
Thus all stated $O(n)$-time and $O(n)$-space bounds hold on a word RAM with word size $\Theta(\log n)$.

\section{The generic regime}\label{sec:generic}

This section proves \cref{case:generic} of \cref{thm:main}.
We first prove the statement by induction and then extract a linear-time construction from the proof.

\begin{theorem}\label{thm:generic}
  Let $a,b$ be distinct real numbers such that $ab(a+b)\neq0$.
  Every tree other than $K_2$ admits a proper $\{a,b\}$-edge-weighting.
\end{theorem}

\begin{proof}
  We argue by induction on the order of $T$.
  An isolated vertex has no edges and therefore vacuously admits a proper $\{a,b\}$-edge-weighting.
  As observed above, $K_2$ admits no proper weighting.
  If $T$ is a star with $m\geq2$ edges, assign weight $a$ to every edge.
  Every leaf then has weighted degree $a$, while the center has weighted degree $ma$.
  Their difference is $(m-1)a\neq0$, so this weighting is proper.

  It remains to consider a tree $T$ that is neither $K_2$ nor a star.
  Root $T$ at a leaf $r$.
  Choose a nonleaf vertex $v$ at maximum distance from $r$.
  Let $u$ be the parent of $v$, and let $L$ be the set of children of $v$.
  By the choice of $v$, every child of $v$ is a leaf.
  Let $k=|L|$.
  Since $v$ is not a leaf, $k\geq1$.

  The smaller tree $T-L$ is not isomorphic to $K_2$.
  If it were, its vertices would be $u$ and $v$, and the root $r$, which remains in $T-L$, would have to be $u$.
  Every vertex of $T$ other than $v$ would then be a leaf adjacent to $v$, making $T$ a star, a contradiction.
  The induction hypothesis therefore gives a proper $\{a,b\}$-edge-weighting $\omega$ of $T-L$.

  For $j\in\{0,\dots,k\}$, extend $\omega$ by assigning weight $b$ to exactly $j$ edges joining $v$ to $L$ and weight $a$ to the other $k-j$ edges.
  The resulting weighted degree of $v$ is
  \begin{equation}\label{eq:Sj}
    S_j=\omega(uv)+(k-j)a+jb.
  \end{equation}
  Since $S_{j+1}-S_j=b-a\neq0$, the values $S_0,\dots,S_k$ are pairwise distinct.
  Among the vertices of $T-L$, only the weighted degree of $v$ changes.
  Thus the extension is proper if and only if
  \begin{equation}\label{eq:new-constraints}
    S_j\neq s_\omega(u),\qquad
    S_j\neq a\ \text{ for }j<k,\qquad
    S_j\neq b\ \text{ for }j>0
  \end{equation}
  hold.
  The second and third inequalities separate $v$ from its new leaf neighbors on edges of weight $a$ and $b$, respectively.
  Since the values $S_j$ are distinct, each forbidden equality in \cref{eq:new-constraints} excludes at most one value of $j$.
  If $k\geq3$, there are at least four candidates, so some $j$ satisfies all three constraints.
  Suppose next that $k=1$.
  For $j=0$, the unique new leaf has weighted degree $a$ and $S_0=\omega(uv)+a$; for $j=1$, it has weighted degree $b$ and $S_1=\omega(uv)+b$.
  In either case, the difference between the two endpoint sums is $\omega(uv)\neq0$.
  Since $S_0\neq S_1$, at most one equals $s_\omega(u)$, and the other choice gives a proper extension.
  Finally, suppose that $k=2$.
  For $j=0$, a conflict with a new leaf would give $\omega(uv)+2a=a$, hence $\omega(uv)=-a$.
  For $j=2$, a conflict with a new leaf would give $\omega(uv)+2b=b$, hence $\omega(uv)=-b$.
  For $j=1$, comparison with an $a$-leaf or a $b$-leaf gives, respectively, $\omega(uv)=-b$ or $\omega(uv)=-a$.
  Thus a conflict with a new leaf would imply $\omega(uv)\in\{-a,-b\}$.
  Since $\omega(uv)\in\{a,b\}$, this would force $a=0$, $b=0$, or $a+b=0$, contrary to the hypotheses.
  Hence every candidate separates $v$ from its new leaf neighbors.
  At most one of the three values equals $s_\omega(u)$, so a proper extension exists.

  In every case, the extended weighting is proper, completing the induction.
\end{proof}

At each extension step, it suffices to choose an admissible value of $j$, so the proof is constructive.
The next corollary implements the same reductions without repeatedly searching for a deepest nonleaf vertex.

\begin{corollary}\label{cor:linear-algorithm}
  Under the hypotheses of \cref{thm:generic}, a proper $\{a,b\}$-edge-weighting of an $n$-vertex tree $T\not\cong K_2$ can be found in $O(n)$ time and $O(n)$ space.
\end{corollary}

\begin{proof}
  Handle the cases $n=1$ and stars directly, as in the proof of \cref{thm:generic}.
  Otherwise, root $T$ at a leaf $r$, and let $r'$ be its neighbor.
  A depth-first search computes the parent and child lists and orders the vertices in postorder.
  Traverse the vertices in postorder.
  Whenever a vertex $v\notin\{r,r'\}$ has any remaining children, record its current child list and delete the corresponding children.
  At that point, every proper descendant of a current child has already been deleted, so each current child is a leaf.
  Each deletion is therefore precisely the reduction used in the proof of \cref{thm:generic}.
  After all such deletions, every remaining vertex other than $r'$ is a leaf adjacent to $r'$, so the remaining tree is a star with at least two edges.
  Assign weight $a$ to all edges of this star.

  Restore the recorded leaf sets in reverse order.
  Inductively, immediately before restoring a record at $v$, the current tree is properly weighted.
  Because records are restored in reverse postorder, the record centered at the parent $u$, if it exists, has already been restored, and no subsequent restoration changes an edge incident with $u$.
  Hence $\omega(uv)$ and $s_\omega(u)$ remain fixed for the rest of the construction.
  Starting from $S_0=\omega(uv)+ka$, test $j=0,\dots,k$ using \cref{eq:new-constraints}.
  After a failed test with $j<k$, obtain $S_{j+1}$ from $S_j$ by adding $b-a$.
  The case analysis in the proof of \cref{thm:generic} shows that one of these tests succeeds.
  For the first successful value of $j$, assign weight $b$ to $j$ of the stored child edges and weight $a$ to the rest, then update the weighted degrees.
  This preserves properness and reestablishes the invariant for the next restoration.

  Each deleted child edge is recorded and restored once, while each edge of the remaining star is assigned once.
  A record with $k$ children requires at most $k+1$ tests, and the sums of the numbers of children and of the numbers of records are both $O(n)$.
  Consequently, the algorithm performs $O(n)$ word operations and uses $O(n)$ space for the rooted tree, records, weights, and weighted degrees.
\end{proof}

\section{The opposite-weight regime}\label{sec:opposite}

This section proves \cref{case:opposite} of \cref{thm:main}.
As noted in the introduction, scaling reduces the argument to the pair $\{-1,1\}$.
For this pair, the subgraph induced by the degree-$3$ vertices characterizes the trees without a proper weighting.
Recall that $\Cfam$ is the family of trees whose vertices have degree $1$ or $3$ and for which the subgraph induced by the degree-$3$ vertices has a perfect matching.
In particular, $K_2\in\Cfam$: both vertices have degree $1$, and $\degthree(K_2)$ is empty and has a perfect matching.

To expose the matching condition, we record the possible weighted degrees at the neighbor of a normalized root leaf.
Let $T$ be a tree with at least two vertices, let $r$ be a leaf, and let $v$ be its neighbor.
Define $\mathcal A(T,r)$ to be the set of values $s_\omega(v)$ attained by $\{-1,1\}$-edge-weightings $\omega$ such that $\omega(rv)=1$ and every edge other than $rv$ is proper.
Thus $\mathcal A(T,r)$ records the attainable values at the other endpoint of a root edge normalized to weight $1$.
Globally negating all edge weights preserves properness, so every proper weighting can be normalized to satisfy $\omega(rv)=1$.
The root leaf has weighted degree $1$, so, once nonemptiness is known, $T$ has no proper weighting exactly when $\mathcal A(T,r)=\{1\}$.
The next two lemmas establish nonemptiness and characterize the singleton states $\{1\}$ and $\{3\}$, which translate into perfect-matching conditions on the subgraph induced by the degree-$3$ vertices.

For the local calculation, represent branch $i$ at $v$ by a normalized attainable value $q_i$, defined as the weighted degree at its attachment vertex when the incident edge has weight $1$.
The rooted construction formalizing this normalization appears in \cref{lem:rooted-opposite}.
Multiplying all weights in branch $i$ by $\sigma_i\in\{-1,1\}$ changes this value to $\sigma_i q_i$, while the weighted degree at $v$ becomes $1+\sum_i\sigma_i$.
We must therefore choose the signs so that the latter value differs from every $\sigma_i q_i$.
The next lemma proves that such a choice always exists and characterizes the two singleton outcomes.

\begin{lemma}\label{lem:local-signing}
  For integers $q_1,\dots,q_k$ and $\sigma=(\sigma_1,\dots,\sigma_k)\in\{-1,1\}^k$, set $S(\sigma)=1+\sum_{i=1}^k\sigma_i$, and let
  \[
    \Phi(q_1,\dots,q_k)
      =\left\{
        S(\sigma):
        \sigma\in\{-1,1\}^k,
        \ S(\sigma)\neq\sigma_i q_i
        \text{ for every }i
      \right\}.
  \]
  Then $\Phi(q_1,\dots,q_k)$ is nonempty.
  Moreover,
  \begin{enumerate}[label=(\roman*),ref=(\roman*),leftmargin=*]
    \item\label{item:phi-one} $\Phi(q_1,\dots,q_k)=\{1\}$ if and only if either $k=0$, or $k=2$ and $\{q_1,q_2\}=\{1,3\}$, and
    \item\label{item:phi-three} $\Phi(q_1,\dots,q_k)=\{3\}$ if and only if $k=2$ and $q_1=q_2=1$.
  \end{enumerate}
\end{lemma}

\begin{proof}
  Put $N_t=\{i:q_i=t\}$ and $n_t=|N_t|$, and suppose that exactly $m$ of the signs $\sigma_i$ are negative.
  The resulting sum is $S_m=k+1-2m$.
  Suppose first that $S_m\neq0$.
  If $q_i=S_m$, then choosing $\sigma_i=1$ would give $S_m=\sigma_i q_i$, so $\sigma_i$ must be $-1$.
  If $q_i=-S_m$, then choosing $\sigma_i=-1$ would give the same conflict, so $\sigma_i$ must be $1$.
  All other indices are unrestricted.
  Once exactly $m$ signs are negative, these restrictions are summarized by the following table.
  \[
    \begin{array}{c|c}
      \text{value of }q_i & \text{restriction on }\sigma_i\\
      \hline
      S_m & \sigma_i=-1\\
      -S_m & \sigma_i=1\\
      q_i\notin\{S_m,-S_m\} & \text{unrestricted}
    \end{array}
  \]
  Hence an admissible choice with exactly $m$ negative signs exists if and only if
  \begin{equation}\label{eq:signing-criterion}
    n_{S_m}\leq m
    \qquad\text{and}\qquad
    n_{-S_m}\leq k-m.
  \end{equation}
  Thus, for $S_m\neq0$, the target sum $S_m$ is inadmissible precisely when at least one inequality in \cref{eq:signing-criterion} fails.
  If $S_m=0$, the condition $S_m\neq\sigma_i q_i$ is independent of the sign and holds for every $i$ precisely when $n_0=0$.

  We first prove nonemptiness.
  Choose the smallest $m\in\{0,\dots,k\}$ such that $n_{S_m}\leq m$; such an $m$ exists because $m=k$ works.
  If $S_m=0$ and $n_0>0$, then minimality gives $n_{S_j}\geq j+1$ for $0\leq j<m$.
  Because $S_m=0$, we have $k=2m-1$.
  The values $S_0,\dots,S_{m-1},0$ are distinct, and hence
  \[
    k\geq 1+\sum_{j=0}^{m-1}(j+1)
      =1+\frac{m(m+1)}2>2m-1=k,
  \]
  a contradiction.
  Thus $n_0=0$, so the sum $S_m=0$ is admissible in this case.

  Suppose now that $S_m\neq0$.
  If the second inequality in \cref{eq:signing-criterion} holds, then this choice of $m$ gives an admissible signing.
  It remains to consider the case in which that inequality fails.
  Then $m\geq1$.
  By minimality, the first inequality fails for $m-1$; since $S_{m-1}=S_m+2$, this gives $n_{S_m+2}\geq m$.
  The failure of the second inequality for $m$ gives $n_{-S_m}\geq k-m+1$.
  If $S_m+2\neq-S_m$, the sets $N_{S_m+2}$ and $N_{-S_m}$ are disjoint and contain at least $m+(k-m+1)=k+1$ indices, which is impossible.
  Hence $S_m+2=-S_m$, and therefore $S_m=-1$.
  Since $S_{m-1}=1$, the failure of the first inequality at $m-1$ gives $n_1\geq m$.
  Moreover, $k$ is even and $m=k/2+1$.
  Since $m\leq k$, this implies $k\geq2$, so $m'=m-2=k/2-1$ lies in $\{0,\dots,k\}$.
  For this value, $S_{m'}=3$ and $n_3,n_{-3}\leq k-n_1\leq m'$.
  Hence \cref{eq:signing-criterion} holds for $m'$.
  Thus $\Phi(q_1,\dots,q_k)$ is nonempty.

  We next characterize the singleton set $\{1\}$.
  Suppose that $\Phi(q_1,\dots,q_k)=\{1\}$ and $k>0$.
  Parity forces $k$ to be even.
  For the sum $1$, the number of negative signs is $k/2$, so admissibility and \cref{eq:signing-criterion} give $n_1,n_{-1}\leq k/2$.
  For the sum $-1$, the number of negative signs is $k/2+1$.
  Its first inequality is automatic from $n_{-1}\leq k/2$, so inadmissibility forces $n_1\geq k/2$.
  Hence $n_1=k/2$.
  If $k\geq4$, then the sum $3$ uses $k/2-1$ negative signs, and its inadmissibility implies either $n_3\geq k/2$ or $n_{-3}\geq k/2+2$.
  The latter is impossible because the $n_1=k/2$ indices with value $1$ are disjoint from $N_{-3}$.
  In the former case, the disjoint sets $N_1$ and $N_3$ together contain all $k$ indices, so no index satisfies $q_i=k+1$.
  The all-positive signing would then satisfy \cref{eq:signing-criterion} for its sum $k+1\geq5$, a contradiction.
  Thus $k=2$.
  Inadmissibility of the all-negative sum $-1$ forces one of $q_1,q_2$ to equal $1$, and inadmissibility of the all-positive sum $3$ forces the other to equal $3$.
  Conversely, if $(q_1,q_2)$ is either $(1,3)$ or $(3,1)$, checking the four signings shows that the only admissible sum is $1$.

  It remains to characterize the singleton set $\{3\}$.
  Suppose that $\Phi(q_1,\dots,q_k)=\{3\}$.
  Again $k$ is even and $k\geq2$.
  If $k\geq4$, then admissibility of the sum $3$, which uses $(k-2)/2$ negative signs, gives $n_3\leq(k-2)/2$ and $n_{-3}\leq(k+2)/2$.
  For the sum $-3$, the first inequality is therefore automatic, so its inadmissibility forces $n_3\geq(k-2)/2$ and hence $n_3=(k-2)/2$.
  Since the sum $1$ is inadmissible, one of $n_1,n_{-1}$ is at least $k/2+1$.
  Whichever of $N_1$ and $N_{-1}$ has this size is disjoint from $N_3$, and together the two sets contain at least $k$ indices.
  They therefore occupy all indices, so no index has value $k+1$.
  The all-positive sum $k+1\neq3$ would therefore be admissible, a contradiction.
  Thus $k=2$.
  Inadmissibility of the all-negative sum $-1$ forces one of $q_1,q_2$ to equal $1$.
  Give that index the negative sign and the other index the positive sign.
  The resulting sum is $1$, so its inadmissibility forces the other value to equal $1$ as well.
  Thus $q_1=q_2=1$.
  Conversely, a direct check of the four signings shows that $q_1=q_2=1$ gives the single admissible sum $3$.
\end{proof}

We now apply \cref{lem:local-signing} to $\mathcal A(T,r)$.

\begin{lemma}\label{lem:rooted-opposite}
  Let $T$ be a tree with at least two vertices, let $r$ be a leaf, and let $v$ be its neighbor.
  Then $\mathcal A(T,r)$ is nonempty.
  Furthermore,
  \begin{align}
    \mathcal A(T,r)=\{1\}
      &\ \Longleftrightarrow\ 
      T\in\Cfam,\label{eq:rooted-one}\\
    \mathcal A(T,r)=\{3\}
      &\ \Longleftrightarrow\ 
      d_T(w)\in\{1,3\}\text{ for } w\in V(T),
      \ d_T(v)=3,\text{ and }\degthree(T)-v\text{ has a perfect matching}.
      \label{eq:rooted-three}
  \end{align}
\end{lemma}

\begin{proof}
  We argue by induction on $|V(T)|$.
  We first derive a recurrence for $\mathcal A(T,r)$, then characterize its singleton states $\{1\}$ and $\{3\}$, and finally translate those recurrences into matching conditions on the subgraph induced by the degree-$3$ vertices.
  Let $u_1,\dots,u_k$ be the neighbors of $v$ other than $r$.
  For each $i$, form $T_i$ from the component of $T-v$ containing $u_i$ by adding a new leaf $r_i$ adjacent to $u_i$.
  Each $T_i$ has fewer vertices than $T$, so by induction $\mathcal A(T_i,r_i)$ is nonempty.
  If $k=0$, then $T\cong K_2$; this base case will also follow directly from the recurrence below.

  For each $i$, choose $q_i\in\mathcal A(T_i,r_i)$ and a weighting of $T_i$ that realizes this value.
  Multiply every weight in $T_i$ by a sign $\sigma_i\in\{-1,1\}$.
  The artificial root edge $r_iu_i$ then has weight $\sigma_i$, and the weighted degree at $u_i$ becomes $\sigma_i q_i$.
  Delete $r_i$ and replace its incident edge by $vu_i$ with the same weight $\sigma_i$.
  This replacement preserves the contribution at $u_i$, while the normalized edge $rv$ gives $s_\omega(v)=1+\sum_{i=1}^k\sigma_i$.
  Thus the edge $vu_i$ is proper exactly when $1+\sum_j\sigma_j\neq\sigma_i q_i$.
  Therefore, for every choice of values $q_i\in\mathcal A(T_i,r_i)$ for $i=1,\dots,k$, each value in $\Phi(q_1,\dots,q_k)$ belongs to $\mathcal A(T,r)$.
  For the reverse inclusion, take any weighting counted by $\mathcal A(T,r)$.
  For each $i$, let $\sigma_i=\omega(vu_i)$, add the artificial leaf $r_i$ with edge weight $\sigma_i$, and multiply every weight in the resulting $T_i$ by $\sigma_i$.
  The artificial root edge then has weight $\sigma_i^2=1$, and the multiplication preserves all inequalities on internal edges.
  The normalized weighted degree at $u_i$ is therefore some $q_i\in\mathcal A(T_i,r_i)$, while its value before normalization was $\sigma_i q_i$.
  Hence the attainable values at $v$ are obtained exactly by independently selecting one value from each rooted branch, choosing a weighting that realizes each selected value, and applying the local signing operation $\Phi$.
  Therefore
  \begin{equation}\label{eq:rooted-recurrence}
  \mathcal A(T,r)
    =\bigcup\bigl\{
        \Phi(q_1,\dots,q_k)
        :
        q_i\in\mathcal A(T_i,r_i)
        \text{ for }i=1,\dots,k
      \bigr\}.
\end{equation}
  For $k=0$, the unique weighting has $s_\omega(v)=1$, so $\mathcal A(T,r)=\{1\}$.
  For $k>0$, nonemptiness follows from the induction hypothesis and \cref{lem:local-signing}.

  We now extract from \cref{eq:rooted-recurrence} the conditions for the two singleton states.
  If $\mathcal A(T,r)=\{1\}$, then every nonempty set $\Phi(q_1,\dots,q_k)$ occurring in \cref{eq:rooted-recurrence} equals $\{1\}$.
  By \cref{item:phi-one}, either $k=0$, or $k=2$ and every choice $q_i\in\mathcal A(T_i,r_i)$ satisfies $\{q_1,q_2\}=\{1,3\}$.
  In the latter case, choose $x\in\mathcal A(T_1,r_1)$ and $y\in\mathcal A(T_2,r_2)$.
  After reindexing, assume that $x=1$ and $y=3$.
  For any $x'\in\mathcal A(T_1,r_1)$, the equality $\{x',y\}=\{1,3\}$ and $y=3$ force $x'=1$.
  Similarly, every $y'\in\mathcal A(T_2,r_2)$ equals $3$.
  Thus, after reindexing, the two attainable sets are exactly $\{1\}$ and $\{3\}$.
  Conversely, if the two attainable sets are $\{1\}$ and $\{3\}$, then \cref{item:phi-one} makes every set in the union \cref{eq:rooted-recurrence} equal to $\{1\}$.
  By \cref{item:phi-three}, the equality $\mathcal A(T,r)=\{3\}$ requires $k=2$ and every choice of $q_1,q_2$ to satisfy $q_1=q_2=1$.
  Since both attainable sets are nonempty, this condition is equivalent to $\mathcal A(T_1,r_1)=\mathcal A(T_2,r_2)=\{1\}$.
  Conversely, these two singleton sets make every term in \cref{eq:rooted-recurrence} equal to $\{3\}$ by \cref{item:phi-three}.
  Thus
  \begin{align}
    \mathcal A(T,r)=\{1\} &\ \Longleftrightarrow\ k=0\text{ or, after reindexing, } k=2,\ \mathcal A(T_1,r_1)=\{1\},\ \mathcal A(T_2,r_2)=\{3\},\label{eq:state-one-recurrence}\\
    \mathcal A(T,r)=\{3\} &\ \Longleftrightarrow\ k=2\text{ and } \mathcal A(T_1,r_1)=\mathcal A(T_2,r_2)=\{1\}.\label{eq:state-three-recurrence}
  \end{align}

  The first recurrence has one child of each singleton type, whereas the second has two children of type $\{1\}$.
  In the first, the $\{3\}$-branch contains the matching edge incident with $v$; in the second, deleting $v$ separates the matching problem into the two subgraphs $\degthree(T_1)$ and $\degthree(T_2)$.
  We now show that the recurrences are equivalent to the matching conditions in \cref{eq:rooted-one,eq:rooted-three}.
  For $T\cong K_2$, both sides of \cref{eq:rooted-one} hold because $\degthree(T)$ is empty and therefore has a perfect matching.
  Suppose first that \cref{eq:state-one-recurrence} holds with $k=2$.
  By induction, $T_1$ belongs to $\Cfam$.
  For $T_2$, write $u_2$ for the neighbor of $r_2$.
  Induction also shows that every vertex of $T_2$ has degree $1$ or $3$, that $d_{T_2}(u_2)=3$, and that the forest $\degthree(T_2)-u_2$ has a perfect matching.
  Removing the artificial leaves and attaching both branches to $v$ preserves the degrees of every retained branch vertex and gives $d_T(v)=3$.
  In particular, $vu_2$ is an edge of $\degthree(T)$.
  The union of a perfect matching of $\degthree(T_1)$, a perfect matching of $\degthree(T_2)-u_2$, and the edge $vu_2$ is a perfect matching of $\degthree(T)$.
  Hence $T\in\Cfam$.

  Conversely, suppose that $T\in\Cfam$.
  If $T\cong K_2$, then $k=0$.
  Otherwise, $d_T(v)=3$, so $k=2$.
  Fix a perfect matching $M$ of $\degthree(T)$.
  The vertex $v$ is matched through exactly one of its two remaining branches, say through $vu_2$.
  The restriction of $M$ to $\degthree(T_1)$ is a perfect matching, while the restriction of $M-\{vu_2\}$ to $\degthree(T_2)-u_2$ is a perfect matching.
  Induction and \cref{eq:state-one-recurrence} now give $\mathcal A(T,r)=\{1\}$.
  This proves \cref{eq:rooted-one}.

  Finally, \cref{eq:state-three-recurrence} gives $\mathcal A(T,r)=\{3\}$ exactly when $k=2$ and both $T_i$ belong to $\Cfam$.
  In that case, every vertex of $T$ has degree $1$ or $3$, and $d_T(v)=3$.
  Moreover, $\degthree(T)-v=\degthree(T_1)\mathbin{\dot\cup}\degthree(T_2)$, so the union of perfect matchings of $\degthree(T_1)$ and $\degthree(T_2)$ is a perfect matching of $\degthree(T)-v$.
  Conversely, suppose that every vertex of $T$ has degree $1$ or $3$, $d_T(v)=3$, and $\degthree(T)-v$ has a perfect matching.
  Then $k=2$ and $\degthree(T)-v=\degthree(T_1)\mathbin{\dot\cup}\degthree(T_2)$.
  A perfect matching of $\degthree(T)-v$ therefore restricts to a perfect matching of each $\degthree(T_i)$.
  Thus $T_1,T_2\in\Cfam$; induction and \cref{eq:state-three-recurrence} give $\mathcal A(T,r)=\{3\}$.
  The induction is complete.
\end{proof}

\begin{theorem}\label{thm:opposite}
  A tree has no proper $\{-1,1\}$-edge-weighting if and only if it belongs to $\Cfam$.
\end{theorem}

\begin{proof}
  The one-vertex tree vacuously admits a proper $\{-1,1\}$-edge-weighting and does not belong to $\Cfam$.
  Let $T$ have at least two vertices, choose a leaf $r$, and let $v$ be its neighbor.
  In every weighting counted by $\mathcal A(T,r)$, the weighted degree of $r$ is $1$, and every edge other than $rv$ is already proper.
  Hence such a weighting is proper on $T$ if and only if its value at $v$ differs from $1$.
  Any proper $\{-1,1\}$-edge-weighting of $T$ can also be normalized to satisfy $\omega(rv)=1$.
  If $\omega(rv)=-1$, multiplying every edge weight by $-1$ preserves all inequalities of weighted degrees and changes the root-edge weight to $1$.
  Therefore $T$ admits a proper $\{-1,1\}$-edge-weighting if and only if $\mathcal A(T,r)$ contains a value other than $1$.
  Since $\mathcal A(T,r)$ itself is nonempty by \cref{lem:rooted-opposite}, the tree $T$ has no proper $\{-1,1\}$-edge-weighting exactly when $\mathcal A(T,r)=\{1\}$.
  By \cref{eq:rooted-one}, this is equivalent to $T\in\Cfam$.
\end{proof}

Scaling the weights now gives \cref{case:opposite} for every distinct real pair $a,b$ satisfying $a+b=0$.

Deciding whether a proper weighting exists is simpler than constructing one.
A single degree scan constructs the forest $\degthree(T)$ and detects any vertex of degree outside $\{1,3\}$, which certifies that $T\notin\Cfam$.
Otherwise, test whether $\degthree(T)$ has a perfect matching by repeatedly matching a leaf of the current forest to its unique neighbor and deleting both vertices.
In a forest, the edge incident with a leaf belongs to every perfect matching, so deleting the leaf and its neighbor preserves the existence of a perfect matching.
Accept the empty forest and reject a nonempty forest with an isolated vertex.
Maintain the current degrees and a queue containing all vertices of current degree at most $1$.
Each vertex is deleted at most once, and each edge causes only a constant number of degree updates, so the recognition procedure takes $O(n)$ time.
When it shows that $T\notin\Cfam$, the record computation below constructs a weighting.

\begin{corollary}\label{cor:opposite-algorithm}
  Given an $n$-vertex tree $T$, one can decide whether $T$ has a proper $\{-1,1\}$-edge-weighting and, if so, construct one in $O(n)$ time using $O(n)$ space.
\end{corollary}

\begin{proof}
  Handle the one-vertex tree directly.
  For a tree with at least two vertices, use the preceding degree-and-matching scan for recognition.
  If it certifies that $T\in\Cfam$, report that no proper weighting exists.
  Otherwise, \cref{thm:opposite} guarantees a proper weighting, which we construct as follows.
  Root $T$ at a leaf $r$, let $v$ be its neighbor, and process the rooted subtrees from the leaves upward as in the proof of \cref{lem:rooted-opposite}.
  For each rooted subtree, store up to three records: an arbitrary witness, a witness with value different from $1$, and a witness with value different from $3$.
  These three records exist exactly when the attainable set of the rooted subtree is nonempty, contains a value different from $1$, and contains a value different from $3$, respectively.
  A record stores its attainable value, pointers to the selected child records, and the complete local signing, represented by the set of negative indices.
  This is exactly the information used by the singleton tests in \cref{lem:local-signing}.

  Fix a selection of child records with values $q_1,\dots,q_k$.
  For each $m=0,\dots,k$, test \cref{eq:signing-criterion} when $S_m\neq0$ and test $n_0=0$ when $S_m=0$.
  In $O(k)$ time, the scan finds an admissible signing and retains the first admissible sum, the first admissible sum different from $1$, and the first admissible sum different from $3$.
  Every stored value is the weighted degree of a vertex and therefore lies in $[-n,n]$.
  The sets $N_s$ and counts $n_s$ can therefore be stored in arrays indexed by $s+n$.
  Recording the indices touched at the current vertex allows the corresponding array entries to be reset in $O(k)$ time.
  When $S_m\neq0$, select the negative indices by including $N_{S_m}$, excluding $N_{-S_m}$, and filling the remaining positions arbitrarily.
  When $S_m=0$, admissibility means that $N_0$ is empty, so any $m$ indices may be chosen as negative.

  Begin by selecting the arbitrary record for each child.
  The nonemptiness part of \cref{lem:local-signing} guarantees that the scan produces a parent record.
  If it produces no value different from $1$, then \cref{item:phi-one} shows that either $k=0$, or $k=2$ and the selected values are $1$ and $3$.
  In the latter case, first try to replace the selected value $1$ by a different stored value of the child supplying it.
  If no such value exists, try instead to replace the selected value $3$ by a different stored value of its child.
  Either replacement changes the selected pair from $\{1,3\}$, so \cref{item:phi-one} guarantees that a second scan produces a value different from $1$.
  If neither replacement exists, one child has only the value $1$ and the other has only the value $3$, so \cref{eq:state-one-recurrence} shows that the parent has only the value $1$.

  Similarly, if the initial scan produces no value different from $3$, then \cref{item:phi-three} gives $k=2$ and two selected child values equal to $1$.
  If either child has a stored value different from $1$, replace the selected value of that child by such a value and perform one further scan.
  By \cref{item:phi-three}, the new scan produces a value different from $3$.
  If neither child has such a stored value, both children have only the value $1$, so \cref{eq:state-three-recurrence} shows that the parent has only the value $3$.
  Hence a constant number of scans computes all three records at a vertex with $k$ children in $O(k)$ time.

  At the root edge $rv$, select a record whose value at $v$ is different from $1$.
  By the proof of \cref{thm:opposite}, such a record exists exactly when $T$ admits a proper $\{-1,1\}$-edge-weighting.
  Set $\omega(rv)=1$ and backtrack through this record with the initial multiplier $\tau=1$.
  Suppose the accumulated multiplier on a branch is $\tau$ and the record uses the local sign $\sigma_i$ for a child.
  Assign weight $\tau\sigma_i$ to the connecting edge and recurse into that child with multiplier $\tau\sigma_i$.
  Every edge is therefore assigned once, and the reconstructed weighting is exactly the weighting represented by the records.
  Since the sum of the numbers of children is $O(n)$, the total time and space are $O(n)$.
\end{proof}

\section{The zero-weight regime}\label{sec:zero}

This section proves \cref{case:zero} of \cref{thm:main}.
After scaling, every pair with exactly one zero weight reduces to $\{0,1\}$.
For an edge $uv$ of a tree $T$, let $T_{u,v}$ be the component of $T-uv$ containing $u$; this is the branch at $v$ determined by $uv$.
The main device in this regime is global rather than branchwise.
Choosing an even vertex set $P$ determines a unique weighting whose \emph{odd weighted-degree set} is $P$.
Whenever the endpoints of an edge have opposite membership in $P$, their weighted degrees have opposite parity, and that edge is automatically proper.
The constructions therefore choose $P$ so that parity separates the endpoints of every edge outside a small, explicitly selected set.
The odd-cut skeleton identifies the edges on which a conflict may remain, and branch profiles encode the endpoint sums under the relevant parity weightings.
A matching-and-load reformulation exposes the connection with the opposite-weight regime, and a charging argument on the tree of skeleton components proves sufficiency.

\subsection{Prescribed weighted-degree parity}\label{sec:zero-parity}

The following lemma constructs the unique weighting with a prescribed set of odd weighted degrees.

\begin{lemma}\label{lem:prescribed-parity}
  Let $T$ be a tree, and let $P\subseteq V(T)$ have even cardinality.
  There is a unique $\{0,1\}$-edge-weighting $\omega_P$ of $T$ for which $s_{\omega_P}(v)$ is odd exactly when $v\in P$.
  Moreover, if $e\in E(T)$ and $A$ is either component of $T-e$, then
  \begin{equation}\label{eq:prescribed-parity-cut}
    \omega_P(e)\equiv |A\cap P|\pmod2.
  \end{equation}
\end{lemma}

\begin{proof}
  Define $\omega_P(e)$ by \cref{eq:prescribed-parity-cut}.
  Because $|P|$ is even, the two components of $T-e$ contain numbers of vertices of $P$ with the same parity.
  Hence the definition is independent of the chosen component.
  For a vertex $v$, the branches $T_{u,v}$ with $u\in N_T(v)$ partition $V(T)\setminus\{v\}$.
  Therefore
  \[
    s_{\omega_P}(v)
      \equiv
      \sum_{u\in N_T(v)}|V(T_{u,v})\cap P|
      \equiv |P\setminus\{v\}|\pmod2,
  \]
  which is odd exactly when $v\in P$.

  Conversely, suppose that a $\{0,1\}$-edge-weighting $\omega$ has odd weighted-degree set $P$.
  If $A$ is a component of $T-e$, then every edge internal to $A$ contributes twice to the sum of the weighted degrees over $A$, while $e$ contributes once.
  Hence
  \[
    |A\cap P|
      \equiv \sum_{v\in A}s_\omega(v)
      \equiv \omega(e)\pmod2.
  \]
  Therefore \cref{eq:prescribed-parity-cut} determines every edge weight, proving uniqueness.
\end{proof}

\subsection{Branch profiles and the skeleton}\label{sec:zero-skeleton}

Let $T$ have bipartition $X\cup Y$.
For $v\in V(T)$, let $\side_T(v)$ be the bipartition class containing $v$ and let $\overline{\side}_T(v)$ be the other class.
Define
\begin{align*}
  \alpha_T(v)
    &=\#\bigl\{u\in N_T(v):
      |V(T_{u,v})\cap\overline{\side}_T(v)|\equiv1\pmod2
    \bigr\},\\
  \beta_T(v)
    &=\#\bigl\{u\in N_T(v):
      |V(T_{u,v})\cap\side_T(v)|\equiv1\pmod2
    \bigr\}.
\end{align*}
Thus the branch profile of $v$ is $\pi_T(v)=\bigl(\alpha_T(v),\beta_T(v)\bigr)$.
Recall also that $J(T)$ is the spanning forest consisting of the edges $e\in E(T)$ for which both components of $T-e$ have odd order.

Assume throughout this subsection that $|X|$ and $|Y|$ are odd.
For an edge $uv\in E(J(T))$, consider the branch $T_{u,v}$ at $v$ determined by this edge.
The parities of $|V(T_{u,v})\cap\overline{\side}_T(v)|$ and $|V(T_{u,v})\cap\side_T(v)|$ are either $(1,0)$ or $(0,1)$.
Call $uv$ an \emph{opposite-parity edge} (an \emph{$\mathsf O$-edge}) in the first case and a \emph{same-parity edge} (an \emph{$\mathsf S$-edge}) in the second.
The classification is independent of the chosen endpoint, as verified in the proof of \cref{lem:profile-parity}.
For a nonskeleton edge $e$, one component of $T-e$ has both bipartition classes odd and the other has both bipartition classes even.
Orient every such edge toward its odd--odd component.
For a vertex $v$, let $d_T^{\mathsf O}(v)$ and $d_T^{\mathsf S}(v)$ be the numbers of incident $\mathsf O$- and $\mathsf S$-edges of $J(T)$, respectively.
Recall that $\ell_T(v)$ is the number of nonskeleton edges directed away from $v$, and that $\lambda_T(v)=d_{J(T)}(v)+2\ell_T(v)$.

The path $P_6$ in \cref{fig:zero-path-example} illustrates the two skeleton-edge types, the orientation of the nonskeleton edges, and the adjusted load.
It also shows that $J(P_6)$ may be disconnected, so load balance is imposed componentwise.
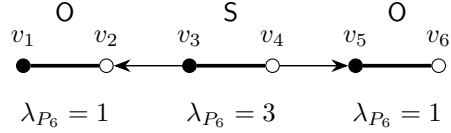
\begin{figure}[ht]
  \centering
  \begin{tikzpicture}[
      xvertex/.style={circle,draw=black,fill=black,inner sep=1.8pt},
      yvertex/.style={circle,draw=black,fill=white,inner sep=1.8pt}
    ]
    \node[xvertex] (v1) at (0,0) {};
    \node[yvertex] (v2) at (1.1,0) {};
    \node[xvertex] (v3) at (2.2,0) {};
    \node[yvertex] (v4) at (3.3,0) {};
    \node[xvertex] (v5) at (4.4,0) {};
    \node[yvertex] (v6) at (5.5,0) {};

    \draw[line width=1.4pt] (v1)--(v2);
    \draw[-{Stealth[length=2.25mm]},line width=.55pt] (v3)--(v2);
    \draw[line width=1.4pt] (v3)--(v4);
    \draw[-{Stealth[length=2.25mm]},line width=.55pt] (v4)--(v5);
    \draw[line width=1.4pt] (v5)--(v6);

    \node[above=4pt] at (v1) {$v_1$};
    \node[above=4pt] at (v2) {$v_2$};
    \node[above=4pt] at (v3) {$v_3$};
    \node[above=4pt] at (v4) {$v_4$};
    \node[above=4pt] at (v5) {$v_5$};
    \node[above=4pt] at (v6) {$v_6$};
    \node at (.55,.72) {$\mathsf O$};
    \node at (2.75,.72) {$\mathsf S$};
    \node at (4.95,.72) {$\mathsf O$};
    \node at (.55,-.65) {$\lambda_{P_6}=1$};
    \node at (2.75,-.65) {$\lambda_{P_6}=3$};
    \node at (4.95,-.65) {$\lambda_{P_6}=1$};
  \end{tikzpicture}
  \caption{The odd-cut skeleton and adjusted loads of $P_6$.
    Filled and unfilled vertices represent the two bipartition classes.
    Thick edges form $J(P_6)$, and the arrows on the nonskeleton edges point toward the odd--odd component.
    The adjusted load shown below each component of $J(P_6)$ is common to its two vertices.}
  \label{fig:zero-path-example}
\end{figure}

Each component of $J(P_6)$ is a copy of $K_2$, and the displayed adjusted load is constant on each component.
Since both bipartition classes have order $3$, \cref{case:zero} of \cref{thm:main} shows that $P_6$ has no proper $\{0,1\}$-edge-weighting.

\begin{lemma}\label{lem:profile-parity}
  If $|X|$ and $|Y|$ are odd, then
  \begin{equation}\label{eq:profile-osb}
    \pi_T(v)=\bigl(d_T^{\mathsf O}(v)+\ell_T(v),d_T^{\mathsf S}(v)+\ell_T(v)\bigr)
  \end{equation}
  for every vertex $v$.
  Moreover, $\alpha_T(v)\equiv1\pmod2$ and $\beta_T(v)\equiv0\pmod2$, and consequently $d_{J(T)}(v)$ is odd.
\end{lemma}

\begin{proof}
  Let $xy\in E(J(T))$, with $x\in X$ and $y\in Y$, and let $A$ and $B$ be the components of $T-xy$ containing $x$ and $y$, respectively.
  Both $A$ and $B$ have odd order, so each has one bipartition class of odd order and one of even order.
  To compare the types at the two endpoints, first suppose that the branch $B$ at $x$ contributes the parity pair $(1,0)$.
  Then $|B\cap Y|$ is odd and $|B\cap X|$ is even.
  Since $|X|$ and $|Y|$ are odd, $|A\cap X|$ is odd and $|A\cap Y|$ is even.
  Relative to the endpoint $y$, the opposite class is $X$ and the same class is $Y$, so the branch $A$ at $y$ also has pair $(1,0)$.
  If the pair at $x$ is $(0,1)$, then $|B\cap Y|$ is even and $|B\cap X|$ is odd; the oddness of $|X|$ and $|Y|$ gives the same pair $(0,1)$ for the branch $A$ at $y$.
  Hence the type $\mathsf O$ or $\mathsf S$ is independent of the endpoint.

  If $e\notin E(J(T))$, the two components of $T-e$ cannot both have odd order.
  Since $|V(T)|$ is even, both components therefore have even order.
  Because both bipartition classes of $T$ are odd, one component has both class sizes odd and the other has both class sizes even.
  The four possible branch contributions are listed below.
  \[
    \begin{array}{l|c}
      \text{incident edge type} & \text{contribution to }(\alpha_T(v),\beta_T(v))\\
      \hline
      \mathsf O\text{-skeleton edge} & (1,0)\\
      \mathsf S\text{-skeleton edge} & (0,1)\\
      \text{nonskeleton edge directed away from $v$} & (1,1)\\
      \text{nonskeleton edge directed toward $v$} & (0,0)
    \end{array}
  \]
  Summing these contributions proves \cref{eq:profile-osb}.

  The branches at $v$ partition $V(T)\setminus\{v\}$.
  Summing the opposite-class sizes of these branches shows that the parity of $\alpha_T(v)$ is $|\overline{\side}_T(v)|$.
  Similarly, summing their same-class sizes shows that the parity of $\beta_T(v)$ is $|\side_T(v)|-1$.
  This proves the two parity identities.
  Finally, \cref{eq:profile-osb} gives $d_{J(T)}(v)=d_T^{\mathsf O}(v)+d_T^{\mathsf S}(v) \equiv \alpha_T(v)+\beta_T(v) \equiv1\pmod2$.
\end{proof}

\subsection{A matching-and-load description}\label{sec:zero-matching}

The next lemma shows that the degree and profile conditions can be expressed using the family $\Cfam$ and the adjusted load.

\begin{lemma}\label{lem:odd-cut-matching-form}
  Let the two bipartition classes of $T$ have odd order, and let $H$ be a component of $J(T)$.
  The following statements are equivalent.
  \begin{enumerate}[label=(\roman*),ref=(\roman*),leftmargin=*]
    \item\label{item:profile-compressed} Every vertex of $H$ has degree $1$ or $3$ in $H$, and $\pi_T$ is constant on $V(H)$.
    \item\label{item:load-compressed} The tree $H$ belongs to $\Cfam$, and $\lambda_T$ is constant on $V(H)$.
  \end{enumerate}
  When these statements hold, the common value of $\lambda_T$ is $2k_H+1$ for an integer $k_H\geq0$, and, for every $v\in V(H)$,
  \begin{equation}\label{eq:odd-cut-load}
    \ell_T(v)=
    \begin{cases}
      k_H   &\text{if } d_H(v)=1,\\
      k_H-1 &\text{if } d_H(v)=3.
    \end{cases}
  \end{equation}
  In particular, $k_H\geq1$ whenever $H$ contains a degree-$3$ vertex.
  Moreover, for every $v\in V(H)$,
  \begin{equation}\label{eq:profile-from-load}
    \pi_T(v)=
    \begin{cases}
      (k_H+1,k_H) &\text{if } k_H\text{ is even},\\
      (k_H,k_H+1) &\text{if } k_H\text{ is odd}.
    \end{cases}
  \end{equation}
\end{lemma}

\begin{proof}
  Suppose first that \cref{item:profile-compressed} holds, and write $\alpha_H$ and $\beta_H$ for the common values of $\alpha_T$ and $\beta_T$, respectively.
  By \cref{eq:profile-osb},
  \begin{equation}\label{eq:os-difference}
    d_T^{\mathsf O}(v)-d_T^{\mathsf S}(v)=\alpha_H-\beta_H
  \end{equation}
  for every $v\in V(H)$.
  At a leaf of $H$, the left-hand side is $1$ or $-1$.
  Hence $\alpha_H-\beta_H\in\{1,-1\}$, and every leaf-edge has the same type.
  At a vertex of degree $3$, \cref{eq:os-difference} shows that exactly two incident edges have the leaf-edge type and exactly one has the other type.
  The edges of the other type are therefore incident with every degree-$3$ vertex exactly once and with no leaf.
  They form a perfect matching of $\degthree(H)$, so $H\in\Cfam$.
  Adding the two coordinates in \cref{eq:profile-osb} gives $\lambda_T(v)=d_H(v)+2\ell_T(v)=\alpha_H+\beta_H$ for every $v\in V(H)$.
  Thus \cref{item:load-compressed} holds.

  Conversely, suppose that \cref{item:load-compressed} holds, and let $M$ be a perfect matching of $\degthree(H)$.
  Extend $M$ to a spanning subgraph of $H$ by including every remaining vertex as an isolated vertex.
  Since every degree in $H$ is odd, the common value of $\lambda_T$ is a positive odd integer.
  Write it as $2k_H+1$ with $k_H\geq0$.
  Since $M$ is incident with each degree-$3$ vertex exactly once and with no leaf, we have $d_H(v)=1+2d_M(v)$ for every $v\in V(H)$.
  The equality $\lambda_T(v)=2k_H+1$ is therefore equivalent to $\ell_T(v)=k_H-d_M(v)$, which is \cref{eq:odd-cut-load}.

  We determine the types of the edges of $H$ from this load equation.
  Let $e=xy\in E(H)$ with $x\in X$ and $y\in Y$, and let $A$ be the component of $H-e$ containing $x$.
  Since every degree in $H$ is odd and $e$ is the only edge of $H$ leaving $A$,
  \[
    |A|\equiv\sum_{v\in A}d_H(v)\equiv1\pmod2.
  \]
  The cut edge is incident with $x\in X$, so
  \[
    \sum_{v\in A\cap Y}d_H(v)=|E(H[A])|=|A|-1.
  \]
  It follows that $|A\cap X|$ is odd and $|A\cap Y|$ is even.
  The component of $T-e$ containing $x$ is obtained from $A$ by adjoining the nonskeleton branches attached to its vertices.
  Such a branch contains an odd number of vertices from each bipartition class precisely when its edge is directed away from its endpoint in $A$; otherwise, both numbers are even.
  Thus a nonskeleton branch whose edge is directed away from its endpoint in $A$ toggles both bipartition parities, whereas one whose edge is directed toward its endpoint in $A$ toggles neither.
  It follows that $e$ is an $\mathsf O$-edge exactly when $\sum_{v\in A}\ell_T(v)$ is even.
  Using \cref{eq:odd-cut-load} and the oddness of $|A|$, we obtain
  \[
    \sum_{v\in A}\ell_T(v)
      \equiv k_H+\sum_{v\in A}d_M(v)
      \equiv k_H+\chi_{\{e\in M\}}\pmod2,
  \]
  because each matching edge internal to $A$ contributes twice to the degree sum, while $e$ is the only edge of $H$ leaving $A$.
  Consequently, $e$ is an $\mathsf O$-edge if and only if $k_H+\chi_{\{e\in M\}}\equiv0\pmod2$.

  If $k_H$ is even, the nonmatching edges are $\mathsf O$-edges and the matching edges are $\mathsf S$-edges.
  Every leaf is therefore incident with one $\mathsf O$-edge, while every degree-$3$ vertex is incident with two $\mathsf O$-edges and one $\mathsf S$-edge.
  Together with \cref{eq:odd-cut-load,eq:profile-osb}, this gives $\pi_T(v)=(k_H+1,k_H)$ at every vertex of $H$.
  If $k_H$ is odd, the two edge types are reversed, and the same calculation gives $\pi_T(v)=(k_H,k_H+1)$ throughout $H$.
  This proves \cref{item:profile-compressed} and \cref{eq:profile-from-load}.
\end{proof}

\subsection{The characterization theorem}\label{sec:zero-characterization}

We state the characterization in matching-and-load form because it makes the connection with the opposite-weight regime explicit.
In the sufficiency proof, the charging takes place on the auxiliary tree $Q$ obtained by contracting the skeleton components.
The equivalent degree-and-profile form in \cref{lem:odd-cut-matching-form} is used for necessity and for constructing a proper weighting when the stated conditions fail.

\begin{theorem}\label{thm:odd-cut-skeleton}
  Let $T$ be a tree with bipartition $X\cup Y$.
  Then $T$ has no proper $\{0,1\}$-edge-weighting if and only if
  \begin{enumerate}[label=(\roman*),ref=(\roman*),leftmargin=*]
    \item\label{item:odd-cut-parts} $|X|$ and $|Y|$ are odd,
    \item\label{item:odd-cut-blocks} every component $H$ of $J(T)$ belongs to $\Cfam$, and
    \item\label{item:odd-cut-load} $\lambda_T$ is constant on every component of $J(T)$.
  \end{enumerate}
\end{theorem}

\begin{proof}
  We first prove necessity, starting with the parity of the two bipartition classes.
  Suppose that $T$ has no proper $\{0,1\}$-edge-weighting.
  If $|X|$ were even, the weighting $\omega_X$ from \cref{lem:prescribed-parity} would give odd weighted degrees on $X$ and even weighted degrees on $Y$.
  Every edge would then be proper.
  If $|Y|$ were even, the analogous weighting $\omega_Y$ would also be proper on every edge.
  Hence both bipartition classes have odd order, proving \cref{item:odd-cut-parts}.

  We next show that $\pi_T$ is constant on each component of $J(T)$.
  Fix an edge $xy\in E(J(T))$, with $x\in X$ and $y\in Y$, and let $A$ and $B$ be the components of $T-xy$ containing $x$ and $y$, respectively.
  Both $A$ and $B$ have odd order.
  Put $P_X=X\mathbin{\triangle}V(A)$.
  The set $P_X$ has even cardinality, so \cref{lem:prescribed-parity} gives the weighting $\omega_{P_X}$.
  Under this weighting, $xy$ is the only edge that can be improper.
  For every edge other than $xy$, either both endpoints lie in $A$ or both lie outside $A$.
  Its endpoints therefore have opposite membership in $P_X$ and hence have weighted degrees of opposite parity under $\omega_{P_X}$.
  Thus every edge other than $xy$ is proper.
  Since $T$ has no proper $\{0,1\}$-edge-weighting, the edge $xy$ must be improper.

  For an edge incident with $x$ other than $xy$, its branch lies in $A$, where $P_X$ agrees with $Y=\overline{\side}_T(x)$.
  For an edge incident with $y$ other than $xy$, its branch lies in $B$, where $P_X$ agrees with $X=\overline{\side}_T(y)$.
  Hence each such edge has weight $1$ exactly when its branch contributes to the first profile coordinate.
  Using the component $A$ in \cref{eq:prescribed-parity-cut}, we have $\omega_{P_X}(xy)\equiv|A\cap Y|\pmod2$.
  This value is $0$ when $xy$ has type $\mathsf O$ and $1$ when it has type $\mathsf S$.
  If $xy$ has type $\mathsf O$, then it contributes to the first profile coordinate at each endpoint but has weight $0$ under $\omega_{P_X}$.
  Hence $s_{\omega_{P_X}}(x)=\alpha_T(x)-1$ and $s_{\omega_{P_X}}(y)=\alpha_T(y)-1$.
  If $xy$ has type $\mathsf S$, then it does not contribute to the first profile coordinate but has weight $1$.
  Hence $s_{\omega_{P_X}}(x)=\alpha_T(x)+1$ and $s_{\omega_{P_X}}(y)=\alpha_T(y)+1$.
  In either case, impropriety of $xy$ gives $\alpha_T(x)=\alpha_T(y)$.

  To compare the second profile coordinates, put $P_Y=Y\mathbin{\triangle}V(A)$.
  This set is even, and the same parity argument again makes $xy$ the only edge that can be improper.
  On every other incident branch at $x$ or $y$, the set $P_Y$ agrees with the bipartition class of the endpoint, so the corresponding edge has weight $1$ exactly when it contributes to the second profile coordinate.
  Using $A$ again in \cref{eq:prescribed-parity-cut}, we have $\omega_{P_Y}(xy)\equiv|A\cap X|\pmod2$.
  This value is $1$ when $xy$ has type $\mathsf O$ and $0$ when it has type $\mathsf S$.
  If $xy$ has type $\mathsf O$, then it does not contribute to the second profile coordinate but has weight $1$ under $\omega_{P_Y}$.
  Hence $s_{\omega_{P_Y}}(x)=\beta_T(x)+1$ and $s_{\omega_{P_Y}}(y)=\beta_T(y)+1$.
  If $xy$ has type $\mathsf S$, then it contributes to the second profile coordinate but has weight $0$.
  Hence $s_{\omega_{P_Y}}(x)=\beta_T(x)-1$ and $s_{\omega_{P_Y}}(y)=\beta_T(y)-1$.
  The edge $xy$ must again be improper, so $\beta_T(x)=\beta_T(y)$.
  Thus $\pi_T(x)=\pi_T(y)$ for every edge $xy\in E(J(T))$.
  Equality along the edges of each component implies that $\pi_T$ is constant on that component.

  We next show that every vertex has degree $1$ or $3$ in $J(T)$.
  By \cref{lem:profile-parity}, every vertex has odd degree in $J(T)$.
  Suppose for a contradiction that a component $H$ of $J(T)$ contains a vertex $v$ with $d=d_H(v)\geq5$.
  The profile equality proved above implies that $\pi_T$ is constant on $V(H)$.
  Write $(\alpha_H,\beta_H)$ for its common value.
  By \cref{eq:profile-osb}, the difference $d_T^{\mathsf O}(w)-d_T^{\mathsf S}(w)$ equals $\alpha_H-\beta_H$ for every $w\in V(H)$.
  The tree $H$ has a leaf, and at a leaf the left-hand side is either $1$ or $-1$.
  Hence $\alpha_H-\beta_H\in\{1,-1\}$.
  Among the $d$ edges of $J(T)$ incident with $v$, one of the two types therefore occurs $(d+1)/2$ times.
  Since $d\geq5$, choose three incident edges of $J(T)$ of this majority type, and denote their set by $C$.

  We construct an even set whose parity weighting is automatically proper outside three selected edges, which we then check directly.
  Let $U$ be the vertex set of the component containing $v$ in $T-C$.
  The three edges of $C$ are precisely the edges with exactly one endpoint in $U$.
  Since $|V(T)|$ is even, every branch corresponding to a nonskeleton edge has even order, while every branch corresponding to an edge of $J(T)$ has odd order.
  The set $U$ contains exactly $d-3$ branches corresponding to edges of $J(T)$ at $v$, and every other included branch has even order.
  Therefore $|U|\equiv1+(d-3)\equiv1\pmod2$, so $|U|$ is odd.

  Put $P=\side_T(v)\mathbin{\triangle}U$.
  Since both $\side_T(v)$ and $U$ have odd cardinality, $P$ has even cardinality.
  Consider the weighting $\omega_P$ from \cref{lem:prescribed-parity}.
  Every edge outside $C$ has both endpoints in $U$ or both endpoints outside $U$.
  Its endpoints therefore have opposite membership in $P$, so every edge outside $C$ is proper under $\omega_P$.

  Suppose first that the three edges of $C$ have type $\mathsf O$.
  For every incident edge of $v$ outside $C$, its branch lies in $U$, where $P$ agrees with $\overline{\side}_T(v)$.
  Such an edge therefore has weight $1$ exactly when its branch contributes to $\alpha_T(v)$, while each edge of $C$ contributes to $\alpha_T(v)$ but has weight $0$.
  Thus $s_{\omega_P}(v)=\alpha_H-3$.
  If $vw\in C$, then every branch at $w$ other than the branch containing $v$ lies outside $U$.
  On each of these branches, $P$ agrees with $\side_T(v)=\overline{\side}_T(w)$, so the corresponding edge has weight $1$ exactly when it contributes to $\alpha_T(w)$.
  As computed from the branch at $v$, the edge $vw$ has weight $0$; it also contributes to $\alpha_T(w)$.
  Since $\alpha_T(w)=\alpha_H$, we have $s_{\omega_P}(w)=\alpha_H-1$.
  Thus every edge of $C$ is proper.

  If the three edges of $C$ have type $\mathsf S$, then at $v$ they do not contribute to $\alpha_T(v)$ but have weight $1$, giving $s_{\omega_P}(v)=\alpha_H+3$.
  For $vw\in C$, the edges at $w$ other than $vw$ are counted as above, while $vw$ does not contribute to $\alpha_T(w)$ but has weight $1$.
  Hence $s_{\omega_P}(w)=\alpha_H+1$.
  Again every edge of $C$ is proper.

  In either case, $\omega_P$ is proper on every edge of $T$, contradicting the assumption.
  Therefore no vertex has degree at least $5$ in $J(T)$.
  Since all degrees in $J(T)$ are odd, every vertex of $J(T)$ has degree $1$ or $3$.

  For every component $H$ of $J(T)$, we have shown that every vertex has degree $1$ or $3$ in $H$ and that $\pi_T$ is constant on $V(H)$.
  By \cref{lem:odd-cut-matching-form}, each such $H$ belongs to $\Cfam$ and $\lambda_T$ is constant on $V(H)$.
  This proves \cref{item:odd-cut-blocks,item:odd-cut-load}.

  We now prove sufficiency.
  Assume \cref{item:odd-cut-parts,item:odd-cut-blocks,item:odd-cut-load}.
  Contract every component of $J(T)$ to one vertex, and let $Q$ be the resulting graph.
  The edges of $Q$ correspond to the nonskeleton edges of $T$ and inherit their orientation toward the odd--odd side.
  No nonskeleton edge has both endpoints in one component of $J(T)$, and no two such edges join the same pair of components, since either situation would create a cycle in $T$.
  Since $T$ is connected and acyclic, $Q$ is a tree.

  Let $\omega:E(T)\to\{0,1\}$ be arbitrary.
  The charging rule below selects a component of $J(T)$ for which nonskeleton edges directed away from the component have weight $1$ and those directed toward it have weight $0$, making their contribution to $s_\omega(v)$ equal to $\ell_T(v)$ at every vertex.
  For an oriented edge of $Q$ from a component $H$ to a component $H'$, charge $H$ when the corresponding edge of $T$ has weight $0$, and charge $H'$ when it has weight $1$.
  Every edge of $Q$ places one charge on exactly one endpoint.
  If $Q$ has one vertex, let $H$ be the component of $J(T)$ represented by that vertex; it receives no charge.
  Otherwise, since $Q$ is a tree, it has $|V(Q)|-1$ edges, so fewer than $|V(Q)|$ charges are distributed among its $|V(Q)|$ vertices.
  Some vertex of $Q$ therefore receives no charge; let $H$ be the corresponding component of $J(T)$.
  Every nonskeleton edge directed away from $H$ then has weight $1$, while every nonskeleton edge directed toward $H$ has weight $0$.
  Consequently, the total contribution of nonskeleton edges to $s_\omega(v)$ is $\ell_T(v)$ for every $v\in V(H)$.

  Define $\eta:E(H)\to\{-1,1\}$ by $\eta(e)=2\omega(e)-1$.
  For every $v\in V(H)$,
  \[
    s_\omega(v)
      =\ell_T(v)+\sum_{e\in E_H(v)}\omega(e)
      =\frac12\lambda_T(v)+\frac12s_\eta(v).
  \]
  The term $\frac12\lambda_T(v)$ is constant on $V(H)$ by \cref{item:odd-cut-load}.
  Since $H\in\Cfam$ by \cref{item:odd-cut-blocks}, \cref{thm:opposite} says that no $\{-1,1\}$-edge-weighting of $H$ is proper.
  In particular, the weighting $\eta$ is not proper, so some edge $uv\in E(H)$ satisfies $s_\eta(u)=s_\eta(v)$.
  The displayed identity then gives $s_\omega(u)=s_\omega(v)$.
  Thus the arbitrary weighting $\omega$ is not proper, and $T$ has no proper $\{0,1\}$-edge-weighting.
\end{proof}

Expanding membership in $\Cfam$ and the adjusted-load condition gives the following leaf--cubic form.

\begin{corollary}\label{cor:odd-cut-matching-characterization}
  Let $T$ be a tree whose two bipartition classes have odd order.
  Then $T$ has no proper $\{0,1\}$-edge-weighting if and only if, for every component $H$ of $J(T)$,
  \begin{enumerate}[label=(\roman*),ref=(\roman*),leftmargin=*]
    \item every vertex of $H$ has degree $1$ or $3$ in $H$,
    \item $\degthree(H)$ has a perfect matching, and
    \item \cref{eq:odd-cut-load} holds for some integer $k_H\geq0$.
  \end{enumerate}
\end{corollary}

\begin{proof}
  The first two conditions state that $H\in\Cfam$.
  Under the degree condition, constancy of $\lambda_T(v)=d_H(v)+2\ell_T(v)$ on $V(H)$ is equivalent to \cref{eq:odd-cut-load}.
  The result therefore follows from \cref{thm:odd-cut-skeleton}.
\end{proof}

\subsection{Consequences and algorithms}\label{sec:zero-consequences}

\begin{corollary}\label{cor:zero-order}
  If a tree $T$ has no proper $\{0,1\}$-edge-weighting and $X\cup Y$ is its bipartition, then $|X|=|Y|\equiv1\pmod2$ and $|V(T)|\equiv2\pmod4$.
\end{corollary}

\begin{proof}
  Let $H$ be a component of $J(T)$.
  By \cref{item:odd-cut-blocks} of \cref{thm:odd-cut-skeleton}, every vertex of $H$ has degree $1$ or $3$ in $H$, and $\degthree(H)$ has a perfect matching.
  For $i\in\{1,3\}$, let $n_i^X$ and $n_i^Y$ be the numbers of degree-$i$ vertices of $H$ in $X$ and $Y$, respectively.
  The perfect matching gives $n_3^X=n_3^Y$.
  Since the degree sums over the two bipartition classes both equal $|E(H)|$, we also have $n_1^X+3n_3^X=n_1^Y+3n_3^Y$.
  Hence $n_1^X=n_1^Y$.
  A tree in which every vertex has degree $1$ or $3$ has two more leaves than degree-$3$ vertices, so $n_1^X+n_1^Y=n_3^X+n_3^Y+2$.
  Together with the preceding equalities, this gives $n_1^X=n_1^Y=n_3^X+1$.
  Both bipartition classes of $H$ therefore have the same odd order $2n_3^X+1$.
  Since $J(T)$ is spanning, summing over its components gives $|X|=|Y|$.
  By \cref{item:odd-cut-parts} of \cref{thm:odd-cut-skeleton}, their common value is odd, which gives the congruence for $|V(T)|$.
\end{proof}

\begin{corollary}\label{cor:odd-cut-algorithm}
  Given an $n$-vertex tree $T$, one can decide whether $T$ has a proper $\{0,1\}$-edge-weighting and, if so, construct one in $O(n)$ time using $O(n)$ space.
\end{corollary}

\begin{proof}
  Root $T$ and compute the numbers of vertices from each bipartition class in every rooted subtree.
  These values determine the parities of the two bipartition classes and, for both directions of every edge, the corresponding branch parities.
  If a bipartition class $Z\in\{X,Y\}$ has even order, set $P=Z$.
  The endpoints of every edge then have opposite membership in $P$, so the weighting $\omega_P$ from \cref{lem:prescribed-parity} is proper.
  One further pass computes $\omega_P$ from \cref{eq:prescribed-parity-cut}, completing this case.

  Assume from now on that both bipartition classes have odd order.
  Construct $J(T)$, orient every nonskeleton edge toward its odd--odd side, and compute $\ell_T(v)$ and $\lambda_T(v)$ for every vertex.
  For each component $H$ of $J(T)$, test whether every degree is $1$ or $3$ and whether $\lambda_T$ is constant on $V(H)$.
  When the degree condition holds, test whether $\degthree(H)$ has a perfect matching using the leaf-matching procedure described before \cref{cor:opposite-algorithm}.
  The degree and load tests process each vertex a constant number of times, and the forests $\degthree(H)$ are vertex-disjoint, so all component tests take total $O(n)$ time.
  If every component passes these tests, \cref{thm:odd-cut-skeleton} shows that no proper $\{0,1\}$-edge-weighting exists.

  It remains to construct a proper weighting when one of these tests fails.
  We use the branch-profile certificates developed in the necessity proof.
  The stored branch parities determine the type of every skeleton edge and all profiles $\pi_T(v)$ in one further scan.
  Retain the first skeleton edge $xy$ with $\pi_T(x)\neq\pi_T(y)$, and retain the first vertex $v$ with $d_{J(T)}(v)\notin\{1,3\}$.
  A failed degree test directly supplies the latter certificate.
  If a component $H$ passes the degree test but fails either the matching test or the adjusted-load test, then \cref{item:load-compressed} of \cref{lem:odd-cut-matching-form} fails.
  By the equivalence in that lemma, and because every degree in $H$ is $1$ or $3$, the profile is not constant on $V(H)$.
  Since $H$ is connected, some edge $xy\in E(H)$ satisfies $\pi_T(x)\neq\pi_T(y)$ and supplies the former certificate.
  Hence failure of a component test produces at least one of the two retained certificates.

  Suppose first that a profile-mismatch edge $xy\in E(J(T))$ was retained, with $x\in X$ and $y\in Y$.
  Let $A$ be the component of $T-xy$ containing $x$.
  If $\alpha_T(x)\neq\alpha_T(y)$, set $P=P_X=X\mathbin{\triangle}V(A)$.
  Otherwise, $\beta_T(x)\neq\beta_T(y)$, so set $P=P_Y=Y\mathbin{\triangle}V(A)$.
  The set $P$ has even cardinality, and every edge other than $xy$ is proper by parity.
  Under $\omega_{P_X}$, the two endpoint sums on $xy$ are $\alpha_T(x)-1$ and $\alpha_T(y)-1$ when $xy$ is an $\mathsf O$-edge, and $\alpha_T(x)+1$ and $\alpha_T(y)+1$ when it is an $\mathsf S$-edge.
  Under $\omega_{P_Y}$, they are $\beta_T(x)+1$ and $\beta_T(y)+1$ when $xy$ is an $\mathsf O$-edge, and $\beta_T(x)-1$ and $\beta_T(y)-1$ when it is an $\mathsf S$-edge.
  The selected coordinate inequality therefore makes $xy$ proper as well.

  Suppose instead that no profile-mismatch edge was retained.
  Then the profile is constant on every component of $J(T)$, and the retained degree certificate is a vertex $v$ in a component $H$ with $d=d_H(v)\notin\{1,3\}$.
  By \cref{lem:profile-parity}, $d$ is odd, so $d\geq5$.
  Write $(\alpha_H,\beta_H)$ for the common profile on $H$.
  If $w$ is a leaf of $H$, then $d_T^{\mathsf O}(w)-d_T^{\mathsf S}(w)$ is $1$ or $-1$.
  By profile constancy and \cref{eq:profile-osb}, the same difference holds at $v$, so one of the two skeleton-edge types occurs $(d+1)/2\geq3$ times there.
  Let $C$ consist of three incident edges of this majority type.
  Let $U$ be the vertex set of the component containing $v$ in $T-C$; one traversal of $T-C$ finds this set.
  Set $P=\side_T(v)\mathbin{\triangle}U$.
  Every included nonskeleton branch has even order, while the $d-3$ included skeleton branches have odd order, so $|U|\equiv1+(d-3)\equiv1\pmod2$.
  Since $\side_T(v)$ has odd cardinality, $P$ has even cardinality.
  Every edge outside $C$ is proper by parity.
  If the edges of $C$ have type $\mathsf O$, then $s_{\omega_P}(v)=\alpha_H-3$ and $s_{\omega_P}(w)=\alpha_H-1$ for each $vw\in C$.
  If they have type $\mathsf S$, the corresponding values are $\alpha_H+3$ and $\alpha_H+1$.
  Thus every edge of $C$ is proper as well.

  For the selected set $P$, one final pass over the rooted tree computes $\omega_P$ from \cref{eq:prescribed-parity-cut} using subtree parities.
  Every scan processes each vertex and edge only a constant number of times, and the matching tests over the components of $J(T)$ take total linear time.
  Hence the total time and space are $O(n)$.
\end{proof}

\begin{proof}[Proof of \cref{thm:main}]
  The three regimes are exhaustive and pairwise disjoint.
  The generic condition excludes the other two, while $a+b=0$ and $ab=0$ together would give $a=b=0$, contrary to the assumption that the weights are distinct.
  Statement \cref{case:generic} is \cref{thm:generic}; after scaling the weights, statements \cref{case:opposite,case:zero} are \cref{thm:opposite,thm:odd-cut-skeleton}, respectively.
\end{proof}

Combining the three constructive arguments gives the corresponding algorithm for each fixed pair of distinct weights.

\begin{corollary}\label{cor:complete-algorithm}
  For every fixed pair of distinct real weights $a,b$, given an $n$-vertex tree $T$, one can decide whether $T$ admits a proper $\{a,b\}$-edge-weighting and, if so, construct one in $O(n)$ time using $O(n)$ space.
\end{corollary}

\begin{proof}
  Apply the construction corresponding to the fixed regime containing $\{a,b\}$.
  In the generic regime, use \cref{cor:linear-algorithm}, with $K_2$ handled directly.
  In the opposite-weight regime, use \cref{cor:opposite-algorithm} and scale the resulting $\{-1,1\}$-edge-weighting.
  In the zero-weight regime, use \cref{cor:odd-cut-algorithm} and scale the resulting $\{0,1\}$-edge-weighting.
  Scaling the weights preserves properness.
  Any final scaling takes constant time per edge.
\end{proof}

\section*{Acknowledgement}

This work was supported by the Ministry of Innovation and Technology NRDI Office within the framework of the Artificial Intelligence National Laboratory Program; by the Ministry of Innovation and Technology of Hungary from the National Research, Development and Innovation Fund, financed under the ELTE TKP 2021-NKTA-62 funding scheme; and by the grant NKFIH-154121.

\bibliographystyle{plain}
\bibliography{bibliography}

\end{document}